\documentclass[a4paper,UKenglish,cleveref,hyperref,autoref]{lipics-v2021}

\title{Twin-width can be exponential in treewidth}
\titlerunning{Twin-width can be exponential in treewidth}

\author{\'{E}douard Bonnet}{Univ Lyon, CNRS, ENS de Lyon, Université Claude Bernard Lyon 1, LIP UMR5668, France \and \url{http://perso.ens-lyon.fr/edouard.bonnet/}}{edouard.bonnet@ens-lyon.fr}{https://orcid.org/0000-0002-1653-5822}{}
\author{Hugues Déprés}{Univ Lyon, CNRS, ENS de Lyon, Université Claude Bernard Lyon 1, LIP UMR5668, France \and \url{http://perso.ens-lyon.fr/hugues.depres/}}{hugues.depres@ens-lyon.fr}{}{}

\authorrunning{\'E. Bonnet, H. Déprés}

\Copyright{Édouard Bonnet, Hugues Déprés}

\category{}

\relatedversion{}

\supplement{}

\acknowledgements{}

\nolinenumbers 

\hideLIPIcs  

\EventEditors{John Q. Open and Joan R. Access}
\EventNoEds{2}
\EventLongTitle{42nd Conference on Very Important Topics (CVIT 2016)}
\EventShortTitle{CVIT 2016}
\EventAcronym{CVIT}
\EventYear{2016}
\EventDate{December 24--27, 2016}
\EventLocation{Little Whinging, United Kingdom}
\EventLogo{}
\SeriesVolume{42}
\ArticleNo{23}

\usepackage[utf8]{inputenc}  

\usepackage[T1]{fontenc}
\usepackage{lmodern}

\usepackage[colorinlistoftodos,bordercolor=orange,backgroundcolor=orange!20,linecolor=orange,textsize=normalsize]{todonotes}

\usepackage{amsmath}  
\usepackage{amssymb}     
\usepackage{bbm}
\usepackage{accents}
\usepackage{complexity}

\usepackage{booktabs}
\usepackage{paralist}
\usepackage{fixmath}

\makeatletter
\newtheorem*{rep@theorem}{\rep@title}
\newcommand{\newreptheorem}[2]{%
\newenvironment{rep#1}[1]{%
 \def\rep@title{#2 \ref{##1}}%
 \begin{rep@theorem}}%
 {\end{rep@theorem}}}
\makeatother

\newreptheorem{theorem}{Theorem}
\newreptheorem{lemma}{Lemma}

\usepackage{xspace}
\usepackage{tikz}

\usepackage[ruled,vlined,linesnumbered]{algorithm2e}

\usetikzlibrary{fit}
\usetikzlibrary{arrows}
\usetikzlibrary{patterns}
\usetikzlibrary{calc}
\usetikzlibrary{shapes}
\usetikzlibrary{positioning}
\usetikzlibrary{math}
\usetikzlibrary{shapes.symbols}
\usetikzlibrary{decorations.pathreplacing,calligraphy}
\tikzset{draw half paths/.style 2 args={%
  decoration={show path construction,
    lineto code={
      \draw [#1] (\tikzinputsegmentfirst) -- 
         ($(\tikzinputsegmentfirst)!0.5!(\tikzinputsegmentlast)$);
      \draw [#2] ($(\tikzinputsegmentfirst)!0.5!(\tikzinputsegmentlast)$)
        -- (\tikzinputsegmentlast);
    }
  }, decorate
}}

\usepackage[scr=boondox,scrscaled=1.05]{mathalfa}

\renewcommand{\geq}{\geqslant}
\renewcommand{\leq}{\leqslant}

\newcommand{\prel}{preleaf\xspace}
\newcommand{\prels}{preleaves\xspace}

\newtheorem{question}{Question}

\newenvironment{proofofclaim}{\noindent \textsc{Proof of the Claim:}}{\hfill$\Diamond$\medskip}

\newcommand{\tww}{\text{tww}}
\newcommand{\otww}{\text{otww}}
\newcommand{\tw}{\text{tw}}
\newcommand{\gn}{\text{gn}}
\newcommand{\mxn}{\text{mxn}}
\newcommand{\fvs}{\text{fvs}}

\renewcommand{\P}{{\mathcal P}}

\definecolor{darkblue}{rgb}{0.0, 0.0, 0.50}

\newcommand\abs[1]{\lvert #1\rvert}

\begin{document}

\maketitle

\begin{abstract}
  For any small positive real $\varepsilon$ and integer $t > \frac{1}{\varepsilon}$, we build a graph with a vertex deletion set of size~$t$ to~a~tree, and twin-width greater than~$2^{(1-\varepsilon) t}$.
  In particular, this shows that the twin-width is sometimes exponential in the treewidth, in the so-called oriented twin-width and grid number, and that adding an apex may multiply the twin-width by at least $2-\varepsilon$.
 Except for the one in oriented twin-width, these lower bounds are essentially tight. 
\end{abstract}

\section{Introduction}\label{sec:intro}

Twin-width is a graph parameter introduced by Bonnet, Kim, Thomassé, and Watrigant~\cite{twin-width1}.
It is defined by means of trigraphs.
A~\emph{trigraph} is a graph with some edges colored black, and some colored red.
A~(vertex) \emph{contraction} consists of merging two (non-necessarily adjacent) vertices, say, $u, v$ into a~vertex~$w$, and keeping every edge $wz$ black if and only if $uz$ and $vz$ were previously black edges.
The other edges incident to $w$ become red (if not already), and the rest of the trigraph remains the same.
A~\emph{contraction sequence} of an $n$-vertex graph $G$ is a sequence of trigraphs $G=G_n, \ldots, G_1=K_1$ such that $G_i$ is obtained from $G_{i+1}$ by performing one contraction.
A~\mbox{\emph{$d$-sequence}} is a contraction sequence in which every vertex of every trigraph has at most $d$ red edges incident to it.
The~\emph{twin-width} of $G$, denoted by $\tww(G)$, is then the minimum integer~$d$ such that $G$ admits a $d$-sequence.
\Cref{fig:contraction-sequence} gives an example of a graph with a 2-sequence, i.e., of twin-width at most~2.
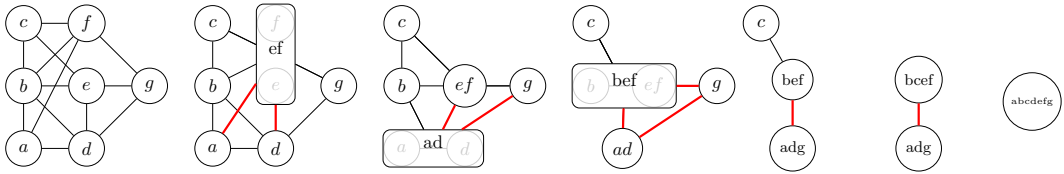
\begin{figure}[h!]
  \centering
  \resizebox{400pt}{!}{
  \begin{tikzpicture}[
      vertex/.style={circle, draw, minimum size=0.68cm}
    ]
    \def\s{1.2}
    \foreach \i/\j/\l in {0/0/a,0/1/b,0/2/c,1/0/d,1/1/e,1/2/f,2/1/g}{
      \node[vertex] (\l) at (\i * \s,\j * \s) {$\l$} ;
    }
    \foreach \i/\j in {a/b,a/d,a/f,b/c,b/d,b/e,b/f,c/e,c/f,d/e,d/g,e/g,f/g}{
      \draw (\i) -- (\j) ;
    }

    \begin{scope}[xshift=3 * \s cm]
    \foreach \i/\j/\l in {0/0/a,0/1/b,0/2/c,1/0/d,2/1/g}{
      \node[vertex] (\l) at (\i * \s,\j * \s) {$\l$} ;
    }
    \foreach \i/\j/\l in {1/1/e,1/2/f}{
      \node[vertex,opacity=0.2] (\l) at (\i * \s,\j * \s) {$\l$} ;
    }
    \node[draw,rounded corners,inner sep=0.01cm,fit=(e) (f)] (ef) {ef} ;
    \foreach \i/\j in {a/b,a/d,b/c,b/d,b/ef,c/ef,c/ef,d/g,ef/g,ef/g}{
      \draw (\i) -- (\j) ;
    }
    \foreach \i/\j in {a/ef,d/ef}{
      \draw[red, very thick] (\i) -- (\j) ;
    }
    \end{scope}

    \begin{scope}[xshift=6 * \s cm]
    \foreach \i/\j/\l in {0/1/b,0/2/c,2/1/g,1/1/ef}{
      \node[vertex] (\l) at (\i * \s,\j * \s) {$\l$} ;
    }
    \foreach \i/\j/\l in {0/0/a,1/0/d}{
      \node[vertex,opacity=0.2] (\l) at (\i * \s,\j * \s) {$\l$} ;
    }
    \draw[opacity=0.2] (a) -- (d) ;
    \node[draw,rounded corners,inner sep=0.01cm,fit=(a) (d)] (ad) {ad} ;
    \foreach \i/\j in {ad/b,b/c,b/ad,b/ef,c/ef,c/ef,ef/g,ef/g}{
      \draw (\i) -- (\j) ;
    }
    \foreach \i/\j in {ad/ef,ad/g}{
      \draw[red, very thick] (\i) -- (\j) ;
    }
    \end{scope}

    \begin{scope}[xshift=9 * \s cm]
    \foreach \i/\j/\l in {0/2/c,2/1/g,0.5/0/ad}{
      \node[vertex] (\l) at (\i * \s,\j * \s) {$\l$} ;
    }
    \foreach \i/\j/\l in {0/1/b,1/1/ef}{
      \node[vertex,opacity=0.2] (\l) at (\i * \s,\j * \s) {$\l$} ;
    }
    \draw[opacity=0.2] (b) -- (ef) ;
    \node[draw,rounded corners,inner sep=0.01cm,fit=(b) (ef)] (bef) {bef} ;
    \foreach \i/\j in {ad/bef,bef/c,bef/ad,c/bef,c/bef,bef/g}{
      \draw (\i) -- (\j) ;
    }
    \foreach \i/\j in {ad/bef,ad/g,bef/g}{
      \draw[red, very thick] (\i) -- (\j) ;
    }
    \end{scope}

    \begin{scope}[xshift=11.7 * \s cm]
    \foreach \i/\j/\l in {0/2/c}{
      \node[vertex] (\l) at (\i * \s,\j * \s) {$\l$} ;
    }
     \foreach \i/\j/\l in {0.5/0/adg,0.5/1.1/bef}{
      \node[vertex] (\l) at (\i * \s,\j * \s) {\footnotesize{\l}} ;
    }
    \foreach \i/\j in {c/bef}{
      \draw (\i) -- (\j) ;
    }
    \foreach \i/\j in {adg/bef}{
      \draw[red, very thick] (\i) -- (\j) ;
    }
    \end{scope}

    \begin{scope}[xshift=13.7 * \s cm]
    \foreach \i/\j/\l in {0.5/0/adg,0.5/1.1/bcef}{
      \node[vertex] (\l) at (\i * \s,\j * \s) {\footnotesize{\l}} ;
    }
    \foreach \i/\j in {adg/bcef}{
      \draw[red, very thick] (\i) -- (\j) ;
    }
    \end{scope}

    \begin{scope}[xshift=15 * \s cm]
    \foreach \i/\j/\l in {1/0.75/abcdefg}{
      \node[vertex] (\l) at (\i * \s,\j * \s) {\tiny{\l}} ;
    }
    \end{scope}
    
  \end{tikzpicture}
  }
  \caption{A 2-sequence witnessing that the initial graph has twin-width at most~2.}
  \label{fig:contraction-sequence}
\end{figure}
Twin-width can be naturally extended to matrices (unordered~\cite{twin-width1} or ordered~\cite{twin-width4}) over a finite alphabet, and hence to any binary structures. 
Classes of binary structures with bounded twin-width include graphs with bounded treewidth, bounded clique-width, $K_t$-minor free graphs, posets with antichains of bounded size, strict subclasses of permutation graphs, map graphs, bounded-degree string graphs~\cite{twin-width1}, segment graphs with no $K_{t,t}$ subgraph, visibility graphs of 1.5D terrains without large half-graphs, visibility graphs of simple polygons without large independent sets~\cite{twin-width8}, as well as $\Omega(\log n)$-subdivisions of $n$-vertex graphs, classes with bounded queue number or bounded stack number, and some classes of cubic expanders~\cite{twin-width2}.

Despite their apparent generality, classes of bounded twin-width are small~\cite{twin-width2}, $\chi$-bounded~\cite{twin-width3}, even quasi-polynomially $\chi$-bounded~\cite{PilipczukS22}, preserved (albeit with a higher upper bound) by first-order transductions~\cite{twin-width1}, and by the usual graph products when one graph has bounded degree~\cite{Pettersson22,twin-width2}, have VC density~1~\cite{tww-polyker,Przybyszewski22}, admit, when $O(1)$-sequences are given, a fixed-parameter tractable first-order model checking~\cite{twin-width1}, an (almost) single-exponential parameterized algorithm for various problems that are W$[1]$-hard in general~\cite{twin-width3}, as well as a~parameterized fully-polynomial linear algorithm for counting triangles~\cite{Kratsch22}, an (almost) linear representation~\cite{PilipczukSZ22}, a stronger regularity lemma~\cite{Przybyszewski22}, etc. 

In all these applications, the upper bound on twin-width, although somewhat hidden in~the previous paragraph, plays a role.
There is then an incentive to obtain as low as possible upper bounds on particular classes of bounded twin-width.
To give one concrete algorithmic example, an independent set of size $k$ can be found in time $O(k^2 d^{2k} n)$ in an $n$-vertex graph given with a $d$-sequence~\cite{twin-width3}.
This is relatively practical for moderate values of $k$, with the guarantee that $d$ is below 10, but not when $d$ is merely upperbounded by $10^{10}$.   
Another motivating example: triangle-free graphs of twin-width at most~$d$ are $d+2$-colorable~\cite{twin-width3}, a~stronger fact in the former case than in the latter.

In that line of work, Balab\'an and Hlinen\'y show that posets of width $k$ (i.e., with antichains of size at most~$k$) have twin-width at most $9k$~\cite{Balaban21}.
Unit interval graphs have twin-width at most~2~\cite{twin-width3}, and proper $k$-mixed-thin graphs (a recently proposed generalization of unit interval graphs) have twin-width $O(k)$~\cite{Balaban22}.
Every graph obtained by subdividing at least $2 \log n$ (throughout the paper, all logs are in base 2) times each edge of an $n$-vertex graph has twin-width at most~4~\cite{Berge21}.
Schidler and Szeider report the (exact) twin-width of a collection of graphs \cite{Schidler21}, obtained via SAT encodings.
Jacob and Pilipczuk~\cite{Jacob22} give the current best upper bound of 183 on the twin-width of planar graphs, while graphs with genus~$g$ have twin-width $O(g)$~\cite{reduced-bdw}.
Most relevant to our paper, for every graph $G$, $\tww(G) \leqslant 3 \cdot 2^{\tw(G)-1}$~\cite{Jacob22}, where $\tw(G)$ denotes the treewidth of~$G$.

Conversely, one may ask the following.
\begin{question}\label{q:tw}
  What is the largest twin-width a graph of treewidth $k$ can have? 
\end{question}
A lower bound of $\Omega(k)$ comes from the existence of $n$-vertex graphs with twin-width $\Omega(n)$ (since the treewidth is trivially upperbounded by $n-1$).
This is almost surely the case of graphs drawn from $G(n,1/2)$.
Alternatively, the $n$-vertex Paley graph (for a prime $n$ such that $n \equiv 1 \mod 4$) has precisely twin-width $(n-1)/2$~\cite{Ahn22}.
Another example to derive the linear lower bound is the power set graph~\cite{Jacob22}.
Improving on this lower bound is not obvious, and $\Theta(k)$ is indeed the answer to~\cref{q:tw} within the class of planar graphs~\cite{Jacob22}, or when replacing 'treewidth' by 'cliquewidth' or 'pathwidth.'

When switching 'twin-width' and 'treewidth' in~\cref{q:tw}, the gap is basically as large as possible:
There are $n$-vertex graphs with treewidth~$\Omega(n)$ and twin-width at most~6, in the iterated 2-lifts of $K_4$~\cite{twin-width2,Bilu06}.

An important characterization of bounded twin-width is via the absence of complex divisions of an adjacency matrix.
A matrix has a \emph{$k$-mixed minor} if its row (resp. column) set can be partitioned into $k$ sets of consecutive rows (resp. columns), such that each of the $k^2$ cells defined by this \emph{$k$-division} contains at least two distinct rows and at least two distinct columns.
The \emph{mixed number of a matrix $M$} is the largest integer $k$ such that $M$ admits a $k$-mixed minor. 
The \emph{mixed number of a graph $G$}, denoted by $\mxn(G)$, is the minimum, taken among all the adjacency matrices $M$ of $G$, of the mixed number of $M$.
The following was shown.
\begin{theorem}[\cite{twin-width1}]\label{thm:mxn}
For every graph $G$, $(\mxn(G) - 1)/2 \leqslant \tww(G) \leqslant 2^{2^{O(\mxn(G))}}$.
\end{theorem}

In sparse graphs (here, excluding a fixed $K_{t,t}$ as a subgraph), the previous theorem is both simpler to formulate and has a better dependency.
A matrix has a \emph{$k$-grid minor} if it has a \emph{$k$-division} with at least one 1-entry in each of its $k^2$ cells.
The \emph{grid number} of a matrix and of a graph $G$, denoted by $\gn(G)$, are defined analogously to the previous paragraph.
We only state the inequality that is useful to bound the twin-width of a sparse class, but is valid in general.
\begin{theorem}[follows from \cite{twin-width1}]\label{thm:gn}
  For every graph $G$, $\tww(G) \leqslant 2^{O(\gn(G))}$.
\end{theorem}
\cref{thm:mxn,thm:gn} allow to bound the twin-width of a class $\mathcal C$ by exhibiting, for every $G \in \mathcal C$, an adjacency matrix of $G$ without large mixed or grid minor.
Therefore one merely has to order $V(G)$ (the vertex set of $G$) in an appropriate way.
The double (resp.~simple) exponential dependency in mixed number (resp.~grid number) implies relatively weak twin-width upper bounds.
For several classes whose twin-width was originally upperbounded via~\cref{thm:mxn}, better bounds were later given by avoiding this theorem (see~\cite{twin-width2,Balaban21,Jacob22,reduced-bdw,Berge21}).
Still for some geometric graph classes, bypassing~\cref{thm:mxn} seems complicated (see~\cite{twin-width8}).
And in general (since this theorem is at the basis of several other applications, see for instance~\cite{twin-width2,twin-width3,twin-width4}) it would help to have an improved upper bound of $\tww(G)$; in particular a negative answer to the following question.

\begin{question}\label{q:mxn-gn}
  Is twin-width sometimes exponential in mixed and grid number?
\end{question}

A variant of twin-width, called \emph{oriented twin-width}, adds an orientation to the red edges (see~\cite{twin-width6}). 
The red edge (arc) is oriented away from the contracted vertex.
The \emph{oriented twin-width}~$d$ of a graph $G$, denoted by $\otww(G)$, is then defined similarly as twin-width by tolerating more than $d$ red arcs incident to a vertex, as long as at most $d$ of them are out-going.
Rather surprisingly twin-width and oriented twin-width are tied.

\begin{theorem}[\cite{twin-width6}]\label{thm:otww}
  For every graph $G$, $\otww(G) \leqslant \tww(G) \leqslant 2^{2^{O(\otww(G))}}$.
\end{theorem}
Classic results show that planar graphs have oriented twin-width at most~9~\cite{twin-width6}.
Thus it would be appreciable to lower the dependency of $\tww(G)$ in $\otww(G)$.

\begin{question}\label{q:otww}
 Is twin-width sometimes exponential in oriented twin-width?
\end{question}

An elementary argument shows that when adding an apex (i.e., an additional vertex with an arbitrary neighborhood) to a graph $G$, the twin-width of the obtained graph is at most $2 \cdot \tww(G) + 1$.
Again it is not clear whether this increase could be made smaller.  

\begin{question}\label{q:apex}
 Does twin-width sometimes essentially double when an apex is added? 
\end{question}

Note that~\cref{q:tw} is asked by Jacob and Pilipczuk~\cite{Jacob22}, and~\cref{q:otww} is posed by Bonnet et al.~\cite{twin-width6}, and is closely related to~\cref{q:mxn-gn}.

\paragraph*{Our contribution.}
With a single construction, we answer all these questions.
The answer to~Questions~\ref{q:mxn-gn}, \ref{q:otww}, and~\ref{q:apex} is affirmative, while the answer to~\cref{q:tw} is $2^{\Theta(k)}$, which confirms the intuition of the authors of~\cite{Jacob22}. 
More precisely, we show the following.
 
\begin{theorem}\label{thm:main}
  For every real $0 < \varepsilon \leqslant 1/2$ and integer $t>1/\varepsilon$, there is a graph $G_{t,\varepsilon}$ with a~feedback vertex set of size $t$ and such that $\tww(G_{t,\varepsilon}) > 2^{(1-\varepsilon)t}$.
\end{theorem}
The graph $G_{t,\varepsilon}$ has in particular treewidth at most $t+1$, grid number at most $t+2$, and oriented twin-width at most $t+1$.
Thus
\begin{compactitem}
\item $\tww(G_{t,\varepsilon}) > 2^{(1-\varepsilon)(\tw(G_{t,\varepsilon}) - 1)}$,
\item $\tww(G_{t,\varepsilon}) > 2^{(1-\varepsilon)(\gn(G_{t,\varepsilon}) - 2)}$, and 
\item $\tww(G_{t,\varepsilon}) > 2^{(1-\varepsilon)(\otww(G_{t,\varepsilon}) - 1)}$.
\end{compactitem}

Hence~\cref{thm:main} has the following consequences.
\begin{corollary}\label{cor:tw}
  For every small $\varepsilon >0$, there is a family $\mathcal F$ of graphs with unbounded twin-width such that for every $G \in \mathcal F$:
  $\tww(G) > 2^{(1-\varepsilon)(\tw(G) - 1)}$.
\end{corollary}
Up to multiplicative factors, it matches the known upper bound~\cite{Jacob22,twin-width1}, and essentially settles~\cref{q:tw}.
The following answers~\cref{q:mxn-gn}.

\begin{corollary}\label{cor:gn}
  For every small $\varepsilon >0$, there is a family $\mathcal F$ of graphs with unbounded twin-width such that for every $G \in \mathcal F$:
  $\tww(G) > 2^{(1-\varepsilon)(\gn(G) - 2)}$.
\end{corollary}

 The following answers~\cref{q:otww}.
\begin{corollary}\label{cor:otww}
  For every small $\varepsilon >0$, there is a family $\mathcal F$ of graphs with unbounded twin-width such that for every $G \in \mathcal F$:
  $\tww(G) > 2^{(1-\varepsilon)(\otww(G) - 1)}$.
\end{corollary}

The following answers~\cref{q:apex}.
\begin{corollary}\label{cor:apex}
  For every small $\varepsilon >0$, there is a family $\mathcal F$ of graphs with unbounded twin-width such that for every $G \in \mathcal F$:
  $\tww(G) > (2-\varepsilon)\tww(G-\{v\})$, where $v$ is a single vertex of $G$.
\end{corollary}

We leave as an open question if the twin-width upper bound in oriented twin-width and mixed number can be made single-exponential.
 


\section{Preliminaries}\label{sec:prelim}

For $i$ and $j$ two integers, we denote by $[i,j]$ the set of integers that are at least $i$ and at most~$j$.
For every integer $i$, $[i]$ is a shorthand for $[1,i]$.
We use the standard graph-theoretic notations: $V(G)$ denotes the vertex set of a graph $G$, $E(G)$ denotes its edge set, $G[S]$ denotes the subgraph of $G$ induced by $S$, etc.
 
We give an alternative approach to contraction sequences.
The \emph{twin-width} of a graph, introduced in~\cite{twin-width1}, can be defined in the following way (complementary to the one given in introduction).
A~\emph{partition sequence} of an $n$-vertex graph $G$, is a sequence $\mathcal P_n, \ldots, \mathcal P_1$ of partitions of its vertex set $V(G)$, such that $\mathcal P_n$ is the set of singletons $\{\{v\}~:~v \in V(G)\}$, $\mathcal P_1$ is the singleton set $\{V(G)\}$, and for every $2 \leqslant i \leqslant n$, $\mathcal P_{i-1}$ is obtained from $\mathcal P_i$ by merging two of its parts into one.
Two parts $P, P'$ of a same partition $\mathcal P$ of $V(G)$ are said \emph{homogeneous} if either every pair of vertices $u \in P, v \in P'$ are non-adjacent, or every pair of vertices $u \in P, v \in P'$ are adjacent.
Two non-homogeneous parts are also said \emph{red-adjacent}.
The \emph{red degree} of a part $P \in \mathcal P$ is the number of other parts of $\mathcal P$ which are red-adjacent to~$P$. 
Finally the twin-width of $G$, denoted by $\tww(G)$, is the least integer $d$ such that there is a partition sequence $\mathcal P_n, \ldots, \mathcal P_1$ of $G$ with every part of every $\mathcal P_i$ ($1 \leqslant i \leqslant n$) having red degree at most~$d$.

The definition of the previous paragraph is equivalent to the one given in introduction, via contraction sequences.
Indeed the trigraph $G_i$ is obtained from partition $\mathcal P_i$, by having one vertex per part of $\mathcal P_i$, a black edge between any fully adjacent pair of parts, and a red edge between red-adjacent parts. 
A \emph{partial contraction sequence} is a sequence of trigraphs $G_n, \ldots, G_i$, for some $i \in [n]$.
A (full) \emph{contraction sequence} is one such that $i=1$.
We naturally consider the trigraph $G_j$ to come \emph{after} (resp.~\emph{before}) $G_{j'}$ if $j<j'$ (resp.~$j>j'$).
Thus when we write \emph{the first trigraph of the sequence $\mathcal S$ to satisfy X} (or \emph{the first time a~trigraph of $\mathcal S$ satisfies X}) we mean the trigraph $G_j$ with largest index $j$ among those satisfying X.
The same goes for partition sequences.

If $u$ is a vertex of a trigraph $H$, then $u(G)$ denotes the set of vertices of $G$ eventually contracted into $u$ in $H$.
We denote by $\P_G(H)$ (and $\P(H)$ when $G$ is clear from the context) the partition $\{u(G) : u \in V(H)\}$ of $V(G)$.
We may refer to a \emph{part} of $H$ as any set in $\{u(G) : u \in V(H)\}$.
We may also refer to a \emph{part} of a contraction/partition sequence as any part of one its trigraphs/partitions.
A contraction \emph{involves} a vertex $v$ if it produces a new part (of size at least~2) containing $v$.
In general, we use trigraphs and partitioned graphs somewhat interchangeably, when one notion appears more convenient than the other. 

\section{Proof of~\cref{thm:main}}

We fix once and for all, $0 < \varepsilon \leqslant 1/2$, a possibly arbitrarily small positive real.
We build for every integer $t > 1/\varepsilon$, a graph $G_{t,\varepsilon}$, that we shorten to $G_t$.
We set $$f(t)=\left\lceil 2+C_t 2^{(1-\varepsilon)t(2+C_t (2^{(1-\varepsilon)t}+1))}\right\rceil$$ where $C_t = 2^{(1-\varepsilon)t}/\varepsilon$.

\medskip
 
 \textbf{Construction of $G_t$.}
 Let $T$ be the full $2^t$-ary tree of depth~$f(t)$, i.e., with root-to-leaf paths on $f(t)$ edges.
 Let $X$ be a set of $t$ vertices, that we may identify to $[t]$.
 The vertex set of $G_t$ is $X \uplus V(T)$.
 The edges of $G_t$ are such that $G[X]$ is an independent set, and $G[V(T)]=T$.
 The edges between $V(T)$ are $X$ are such that
 \begin{compactitem}
 \item the root of $T$ has no neighbor in $X$, and
 \item the $2^t$ children (in $T$) of every internal node of $T$ each have a distinct neighborhood in $X$. 
 \end{compactitem}
 Note that this defines a single graph up to isomorphism.
 By a slight abuse of language, we may utilize the usual vocabulary on trees directly on $G_t$.
 By \emph{root}, \emph{internal node}, \emph{child}, \emph{parent}, \emph{leaf} of $G_t$, we mean the equivalent in $T$.

 \medskip

 We start with this straightforward observation.
\begin{lemma}\label{lem:tw-ub}
 $G_t$ has treewidth at most $t+1$.
\end{lemma}

\begin{proof}\label{lem:tw}
  The set $X$ is a feedback vertex set of $G_t$ of size $t$, thus $\tw(G_t) \leqslant \fvs(G_t)+1 \leqslant t+1$. 
\end{proof}

The following is the core lemma, which occupies us for the remainder of the section.
  \begin{lemma}\label{lem:tww}
   $G_t$ has twin-width greater than $2^{(1-\varepsilon)t}$.
  \end{lemma}

  \begin{proof}
  We assume, by way of contradiction, that $G_t$ admits a $d$-sequence with $d \leq 2^{(1-\varepsilon) t}$.
  We consider the partial $d$-sequence $\mathcal S$, starting at $G_t$, and ending right before the first contraction involving a child of the root.
  We first show that no vertex of $X$ can be involved in a contraction of $\mathcal S$.
  Note that it implies, in particular, that the root cannot be involved in a contraction of $\mathcal S$.
  
  \begin{claim}\label{clm:X}
   No part of $\mathcal S$ contains more than one vertex of $X$. 
  \end{claim}
  
  \begin{proofofclaim}
    Observe that, for every $i \neq j \in [t]$, there are $2^{t-1}$ sets of $2^{[t]}$ containing exactly one of $i, j$: $2^{t-2}$ only contain $i$, and $2^{t-2}$ only contain $j$.
    Recall now that by assumption, in every trigraph of $\mathcal S$, every child of the root is alone in its part.
    Thus a~part $P$ of $\mathcal S$ such that $|P \cap X| \geq 2$ would have red degree at least $2^{t-1} > 2^{(1-\varepsilon) t} \geq d$. 
  \end{proofofclaim}

  \begin{claim}\label{clm:X-T}
    No part of $\mathcal S$ intersects both $X$ and $V(T)$.
  \end{claim}
  
  \begin{proofofclaim}
    For the sake of contradiction, consider the first occurrence of a~part $P \supseteq \{x,v\}$ with $x \in X$ and $v \in V(T)$.
    Vertex $x$ is adjacent to half of the children of the~root, whereas $v$ is adjacent to at most one of them, or all of them (if $v$ is itself the~root).
    In~both cases, this entails at least $2^{t-1}-1$ red edges for $P$ towards children of the~root.
    If~$v$~is not a grandchild of the root, the~red degree of $P$ is at~least $2^{t-1}$.
    We thus assume that~$v$~is a~grandchild of the~root.
    
    As $t \geq 2$, there is a $y \in X \setminus \{x\}$.
    Let $v'$ be the child of $v$ whose neighborhood in $X$ is exactly $\{y\}$.
    This vertex exists since $f(t) \geq 3$.
    If $P$ contains $v'$, $P$ is also red-adjacent to $\{y\}$ (indeed a part, by \cref{clm:X}).
    If instead, $P$ does not contain $v'$, then $P$ is also red-adjacent to the part containing $v'$.
    
    Thus, in any case, the red degree of $P$ is at least $2^{t-1}> 2^{(1-\varepsilon) t} \geq d$.
  \end{proofofclaim}

  From Claims~\ref{clm:X} and \ref{clm:X-T}, we immediately obtain:
  \begin{claim}\label{clm:X-sing}
   Every part of $\mathcal S$ intersecting $X$ is a singleton.
  \end{claim}

  Crucial to the proof, we introduce two properties $\mathscr P$, and later $\mathscr Q$, on internal nodes $v \in V(T)$ in trigraphs $H \in \mathcal S$.
  Property $\mathscr P$ is defined by
  $$\mathscr P(v,H) = \text{\emph{``At least $2^{\varepsilon t}$ children of $v$ are in the same part of $\P(H)$.''}}$$
 
 We first remark that any internal node in a non-singleton part verifies $\mathscr P$.
 \begin{claim}\label{clm:large-children-batch}
   Let $H$ be any trigraph of $\mathcal S$ and $v$ be any internal node of $T$ whose part in $\P(H)$ is not a singleton.
   Then $\mathscr P(v,H)$ holds.
 \end{claim}

 \begin{proofofclaim}
  Let $P$ be the part of $v$ (i.e., the one containing $v$) in $\P(H)$, and $u \in P \setminus \{v\}$. 
  At least $2^t-1$ children of $v$ are not adjacent to $u$.
  Thus these $2^t-1$ vertices have to be in at most $d+1 \leq 2^{(1-\varepsilon) t}+1$ parts. These parts are part $P$, plus at most $d$ parts linked to $P$ by a red edge.
  Since $(2^{\varepsilon t}-1)(2^{(1-\varepsilon) t}+1) < 2^t-1$ (recall that $\varepsilon < 1/2$), one of these parts (possibly $P$) contains at least $2^{\varepsilon t}$ children of $v$. 
 \end{proofofclaim}

 As the merge of a singleton part $\{v\}$ with any other part does not change the intersections of parts with the set of children of $v$, we get a slightly stronger claim.  
  \begin{claim}\label{clm:large-children-batch2}
   Let $v$ be an internal node of $T$, and $H$ be the last trigraph of $\mathcal S$ for which $v$ is in a singleton part of $\mathcal P(H)$.
   Then $\mathscr P(v,H)$ holds.
  \end{claim}
 
 A~\emph{\prel} is an internal node of $T$ adjacent to a leaf, i.e., the parent of some leaves. 
 We obtain the following as a direct consequence of~\cref{clm:large-children-batch}.

\begin{claim}\label{claim:hereditary}
 In any trigraph $H \in \mathcal S$, any non-\prel internal node $v \in V(T)$ that verifies $\mathscr P(v,H)$ has at least $2^{\varepsilon t}$ children $u$ verifying $\mathscr P(u,H)$.
 \end{claim}

We define the property $\mathscr Q$ on internal nodes $v$ of $T$ and trigraphs $H \in \mathcal S$ by induction:
\begin{align*}
  \mathscr Q(v,H) =
    \begin{cases}
      \mathscr P(v,H)                          & \text{\emph{if v is a \prel, and otherwise}}\\
      \mathscr Q(u_1,H) \land \mathscr Q(u_2,H) & \text{\emph{for some pair}}~u_1 \neq u_2~\text{\emph{of children of v.}}
    \end{cases}       
\end{align*}
That is, $\mathscr Q$ is defined as $\mathscr P$ for \prels, and otherwise, $\mathscr Q$ holds when it holds for at least two of its children.
Observe that $\mathscr P$ and $\mathscr Q$ are monotone in the following sense:
If $\mathscr P(v,H)$ (resp.~$\mathscr Q(v,H)$) holds, then $\mathscr P(v,H')$ (resp.~$\mathscr Q(v,H')$) holds for every subsequent trigraph $H'$ of the partial $d$-sequence $\mathcal S$.
We may write that \emph{$v$ satisfies $\mathscr P$ (resp.~$\mathscr Q$) in $H$} when $\mathscr P(v,H)$ (resp.~$\mathscr Q(v,H)$) holds, and may add \emph{for the first time} if no trigraph $H' \in \mathcal S$ before $H$ is such that $\mathscr P(v,H')$ (resp.~$\mathscr Q(v,H')$) holds.  

 \begin{claim}\label{clm:p-implies-q}
 For any trigraph $H \in \mathcal S$ and internal node $v$ of $T$, $\mathscr P(v,H)$ implies $\mathscr Q(v,H)$.
 \end{claim}
 
\begin{proofofclaim}
  This is a tautology if $v$ is a \prel.
  The induction step is ensured by~\cref{claim:hereditary}, since $2^{\varepsilon t} \geq 2$.
\end{proofofclaim}

At the end of the partial $d$-sequence $\mathcal S$, we know, by~\cref{clm:large-children-batch2}, that at least one child of the root satisfies $\mathscr P$, hence satisfies $\mathscr Q$, by~\cref{clm:p-implies-q}.
Thus the first time in the partial $d$-sequence~$\mathcal S$ that $\mathscr Q(v,H)$ holds, for a trigraph $H \in \mathcal S$ and a child $v$ of the root, is well-defined. 
We call $F$ this trigraph, and $v_0$ a child of the root such that $\mathscr Q(v_0,F)$ holds.

We now find many nodes satisfying $\mathscr Q$ in $F$, whose parents form a vertical path of singleton parts.

\begin{claim} \label{claim:branches}
  There is a set $Q \subset V(T)$ of at least $f(t)-2$ internal nodes such that
  \begin{compactitem}
  \item for every $v \in Q$, $\mathscr Q(v,F)$ holds,
  \item the parent of any $v \in Q$ is in a singleton part of $\mathcal P(F)$, and
  \item and \emph{no} two distinct nodes of $Q$ are in an ancestor-descendant relationship.
  \end{compactitem}
\end{claim}

\begin{proofofclaim}
We construct by recurrence two sequences $(v_i)_{i\in [f(t)-2]}, (q_i)_{i\in [0,f(t)-3]}$ of internal nodes of $T$ such that for all $i \in [f(t)-2]$, $v_i$ is a child of $v_{i-1}$, $v_{i-1}$ is in a singleton part of $\P(F)$, and $v_{i-1}$ has a child $q_{i-1} \neq v_i$ for which $\mathscr Q(q_{i-1},F)$ holds.

Assume that the sequence is defined up to $v_i$, for some $i<f(t)-2$.
We will maintain the additional invariant that $v_i$ satisfies $\mathscr Q$ for the first time in $F$.
This is the case for $i=0$.

As $v_i$ is not a \prel, it satisfies $\mathscr Q$ for the first time when a second child of $v_i$ satisfies~$\mathscr Q$.
Let $v_{i+1}$ be this second child, and $q_i$ be the first child to satisfy~$\mathscr Q$ (breaking ties arbitrarily if both children satisfy~$\mathscr Q$ for the first time in $F$).
The vertex $v_{i+1}$ satisfies $\mathscr Q$ for the first time in $F$.
Thus our invariant is preserved.

For every $i \in [f(t)-2]$, $v_i$ is in a singleton part of $\P(F)$.
Indeed, by~\cref{clm:large-children-batch2}, if $v_i$ was not in a singleton part of $\P(F)$, $v_i$ would satisfy $\mathscr P$, hence $\mathscr Q$, in the trigraph preceding $F$; a~contradiction.

The set $Q$ can thus be defined as $\{q_i~:~i \in [0,f(t)-3]\}$.
We already checked that the first two requirements of the lemma are fulfilled.
No pair in $Q$ is in an ancestor-descendant relationship since the nodes of $Q$ are all children of a root-to-leaf path made by the $v_i$s (see~\cref{fig:sequence}). 
\end{proofofclaim}

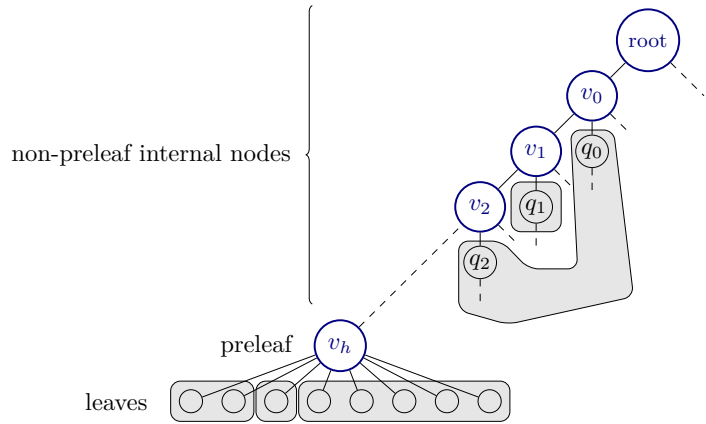
\begin{figure}[ht!]
  \centering
  \resizebox{280pt}{!}{
  \begin{tikzpicture}
    \def\s{0.8}
    \node[draw,thick,circle,darkblue] (r) at (\s,\s) {\footnotesize{root}} ;
    \node[circle,inner sep=0.03cm,] (w) at (2,0) {} ;

    \foreach \i in {0,...,2}{
      \node[draw,thick,circle,darkblue] (v\i) at (-\i * \s,-\s *\i) {$v_\i$} ;
      \node[draw,circle,inner sep=0.03cm,] (q\i) at (-\i * \s,-\s *\i-\s) {$q_\i$} ;
      \node[inner sep=0.3cm] (z\i) at (-\i * \s,-\s *\i-\s) {} ;
    }
    
    \foreach \i in {0,...,2}{
        \draw (v\i) -- (q\i) ;
        \draw[dashed] (v\i) --++(0.7 * \s,-0.7 * \s) ;
        \draw[dashed] (q\i) --++(0,-0.7 * \s) ;
    }
    \draw (v1) -- (v2) ;
    \draw (v0) -- (r) ;
    \draw (v1) -- (v0) ;
    \draw[dashed] (r) --++(\s,-\s) ;
    
    \node[draw,thick,circle, darkblue] (vi) at (-4.5 * \s,- 4.5 * \s) {$v_h$} ;
    \node at (-6 * \s,- 4.5 * \s) {\prel} ;
    \draw [draw, dashed] (vi) -- (v2) ;

    \foreach \i/\j in {1/-3.5,2/-2.5,3/-1.5,4/-0.5,5/0.5,6/1.5,7/2.5,8/3.5}{
      \node[draw,circle] (l\i) at (-4.5 * \s + \j * 0.61,- 5.5 * \s) {} ;
      \draw (l\i) -- (vi) ;
    }

    \foreach \i in {{(l4) (l8)},{(l3)},{(l1) (l2)},{(q1)}}{
    \node[draw,fill, fill opacity=0.1,rounded corners,fit=\i] () {} ;
    }
    
    \node at (-8.5*\s,-5.5*\s) {leaves} ;

    \draw[fill, fill opacity=0.1,rounded corners] (z2.north west) --++ (0,-1) --++ (0.7,-0.2) --++ (1.8,0.4) -- (z0.north east) -- (z0.north west) --++(0.1,-2) --++ (-0.7,0) -- (z2.north east) -- cycle ;

    \draw [decorate, decoration = {brace}] (-4,-3) -- (-4,\s + 0.5);
    \node at (-6.3,-1.25+0.5*\s) {non-\prel internal nodes} ;

  \end{tikzpicture}
  }
  \caption{The nodes $(v_i)_{i \in [0,h]}$ and $(q_i)_{[0,h-1]}$ ($h=f(t)-2$) satisfy $\mathscr P$ and $\mathscr Q$ in $F$. The $v_i$s and the root (nodes circled in blue) are in singleton parts of $F$. The other represented nodes can be in larger parts (shaded areas).}
  \label{fig:sequence}
\end{figure}

Let $B$ the vertices $w \in V(F)$ such that $w(G)$ contains at least $2^{\varepsilon t}$ children of the same node of $T$.
Each vertex of $B$ is red-adjacent to at least $\log(2^{\varepsilon t})=\varepsilon t$ (singleton) parts of $X$.
Therefore, since the red degree of (singleton) parts of $X$ is at most $2^{(1-\varepsilon) t}$: $$|B|\leq \frac{2^{(1-\varepsilon) t}}{\varepsilon}.$$
  
  Next we show that there is relatively large set of vertices of $F$ each corresponding to a~non-singleton part that contains an internal node of $T$.
\begin{claim}
  There is a set $B' \subseteq V(F)$ of size at least $$\frac{1}{(1-\varepsilon) t} \log\left(\frac{f(t)-2}{\abs{B}}\right)-1$$ such that for every $b \in B'$ there is an internal node $v$ of $T$ with $v \in b(G_t)$ and $\abs{b(G_t)} \geqslant 2$.
\end{claim}

\begin{proofofclaim}
  Let $s := \frac{1}{(1-\varepsilon) t} \log(\frac{f(t)-2}{\abs{B}})-1$.
  Our goal is to construct a sequence $(b_i)_{i \in [0,s]}$ of distinct vertices of $F$ such that for every $i \in [s]$,
  \begin{equation}\label{inv:1}
    \text{part}~b_i(G_t)~\text{is \emph{not} a singleton and contains an internal node of}~T.
  \end{equation}
  We first focus on finding $b_0$.
  Note that $b_0$ need not satisfy Invariant~(\ref{inv:1}), but will be chosen to force the existence of $b_1$ itself satisfying (\ref{inv:1}) and starting the induction.
  
  Let $Q := \{q_j~:~0 \leq j \leq f(t)-3\} \subset V(T)$ be as described in~\cref{claim:branches}. 
  Every $q_j \in Q$ has (at least) one descendant $q'_j$ that is a \prel and satisfies $\mathscr Q$, hence $\mathscr P$, in $F$.
  The $q'_j$s are pairwise distinct because no two nodes of $Q$ are in an ancestor-descendant relationship.
  We set $Q' := \{q'_j~:~0 \leq j \leq f(t)-3\}$.

  Now for every $q'_j$, at least $2^{\varepsilon t}$ of its children are in the same part of $\P(F)$; hence, this part corresponds to a vertex in $B$.
  By the pigeonhole principle, there is a $b_0 \in B$ that contains at least $2^{\varepsilon t}$ children of at least $(f(t)-2)/|B|$ nodes of $Q'$.

  For each $b_i$, we define $Q_i \subset Q$ as the set of vertices $q_j$ such that
  \begin{compactitem}
  \item $b_i(G_t)$ contains a (not necessarily strict) descendant $z$ of $q_j$, and
  \item no part $b_{i'}(G_t)$ with $i' < i$ contains a node on the path between $q_j$ and $z$ in $T$.
  \end{compactitem}
  Thus $|Q_0| \geq (f(t)-2)/|B|$.

  We now assume that $b_i \in V(F)$, for some $0 \leq i < s$, has been found with
  \begin{equation}\label{inv:2}
    |Q_i| \geq \frac{f(t)-2}{|B| \cdot 2^{i(1-\varepsilon)t}}.
  \end{equation}
  Observe that $Q_0$ satisfies~(\ref{inv:2}).
  We construct $b_{i+1}, Q_{i+1}$ satisfying the invariants~(\ref{inv:1}) and (\ref{inv:2}).

  For each $q_j \in Q_i$, consider the highest descendant $z_j$ of $q_j$ in $b_i(G_t)$, and $z'_j$ the parent of $z_j$ in $T$.
  By construction, the part $P_j$ of $\P(F)$ containing $z'_j$ is not a $b_k(G_t)$ for any $k \leq i$.
  Part $P_j$ is linked to $b_i(G_t)$ by a red edge.
  Therefore there are at most $2^{(1-\varepsilon) t}$ such parts $P_j$.
  In particular, there is a $b_{i+1} \in V(F)$ such that $b_{i+1}(G_t)$ contains at least
  $$\frac{|Q_i|}{d} \geq \frac{f(t)-2}{|B| \cdot 2^{i(1-\varepsilon) t}} \cdot \frac{1}{2^{(1-\varepsilon) t}} = \frac{f(t)-2}{|B| \cdot 2^{(i+1)(1-\varepsilon) t}}$$
  parents $z'_j$ of highest descendants $z_j$.

  Remark that $b_{i+1}(G_t)$ has size at least two while $(f(t)-2)/(|B| \cdot 2^{(i+1)(1-\varepsilon) t})>1$, which holds since $i<s$.
  Thus $b_{i+1}(G_t)$ does not contain any parent $v_j$ of a $q_j$ (since the $v_j$s are in singleton parts).
  In particular, $|Q_{i+1}| \geq (f(t)-2)/(|B| \cdot 2^{(i+1)(1-\varepsilon) t})$, and $b_{i+1}, Q_{i+1}$ satisfy~(\ref{inv:1}) and (\ref{inv:2}).

  Finally, the set $B':=\{b_i~:~1 \leq i \leq s\}$ has the required properties.
\end{proofofclaim}

We can now finish the proof of the lemma.

  For every $b_i \in B'$, let $u_i \in b_i(G_t)$ be an internal node of $T$.
  As $b_i(G_t) \geq 2$, $u_i$ satisfies $\mathscr P$ in $F$.
  This implies that $b_i$ or a red neighbor of $b_i$ is in $B$.
  Therefore, the total number of red edges incident to a vertex of $B$ is at least $|B'|-|B|$.
  Thus there is a vertex in $B$ with red degree at least $(|B'|-|B|)/|B|$.
  This is a contradiction since

    $$\frac{|B'|-|B|}{|B|} = \frac{|B'|}{|B|}-1 \geq \left(\frac{1}{(1-\varepsilon) t} \log\left(\frac{f(t)-2}{\abs{B}}\right)-1\right) \cdot \frac{1}{|B|}-1 $$
  $$ \geq \left(\frac{1}{(1-\varepsilon) t} \log\left(2^{(1-\varepsilon) t(2+C_t \cdot (2^{(1-\varepsilon) t}+1))}\right)-1\right) \cdot \frac{1}{|B|}-1 $$
  $$ = \left((2+C_t \cdot (2^{(1-\varepsilon) t}+1)) -1\right) \cdot \frac{1}{|B|}-1 > 2^{(1-\varepsilon) t}+1 - 1 = 2^{(1-\varepsilon) t} \geq d.$$
    
    since, we recall, $f(t)=\left\lceil2+C_t \cdot 2^{(1-\varepsilon) t(2+C_t \cdot (2^{(1-\varepsilon) t}+1))}\right\rceil$ and $C_t = \dfrac{2^{(1-\varepsilon) t}}{\varepsilon} \geq |B|$.
\end{proof}

Since $X$ is a feedback vertex set of size $t$ of $G_t$, \cref{lem:tww} implies~\cref{thm:main}, and hence~\cref{cor:tw}.

As the twin-width of $T$ is 2, adding the $t$ apices in $X$, multiplies the twin-width by at~least $2^{t(1-\varepsilon-\frac{1}{t})}$.
Thus one apex in $X$ multiplies the twin-width by at~least $2^{1-\varepsilon-\frac{1}{t}}$, which can be made arbitrarily close to 2.
This establishes~\cref{cor:apex}.

\section{Oriented twin-width and grid number}

In this section, we check that $G_t$ has oriented twin-width at most $t+1$, and grid number at~most~$t+2$.

A~\emph{(partial) oriented contraction sequence} is defined similarly as a \emph{(partial) contraction sequence} with every red edge replaced by a red arc leaving the newly contracted vertex.
Then a \emph{(partial) oriented $d$-sequence} is such that all the vertices of all its \emph{ditrigraphs} have at~most $d$ out-going red arcs.
The~\emph{oriented twin-width} of a graph $G$, denoted by $\otww(G)$, is the~minimum integer $d$ such that $G$ admits an oriented $d$-sequence.

\begin{lemma}\label{lem:otww}
 The oriented twin-width of $G_t$ is at most $t+1$.
\end{lemma}
\begin{proof}
  We observe that the 2-sequence for trees~\cite{twin-width1} is an oriented 1-sequence.
  We contract $T$ to a single vertex (without touching $X$) in that manner.
  This yields a partial oriented $t+1$-sequence for $G_t$ ending on a $t+1$-vertex ditrigraph, which can be contracted in any way.
  This contraction sequence witnesses that $\otww(G_t) \leq t+1$.
\end{proof}
Thus~\cref{cor:otww} holds.

\medskip

We finish by establishing~\cref{cor:gn}.
\begin{lemma}\label{lem:gn}
 The grid number of $G_t$ is at most $t+2$.
\end{lemma}

\begin{proof}
 Recall that $V(G_t) = X \uplus V(T)$.
 Let $\prec$ be the total order on $V(G_t)$ that puts first all the vertices of $X$ in any order, then from left to right, all the~leaves of $T$, followed by the~\prels, the nodes at depth $f(t)-2$, the nodes at depth $f(t)-3$, and so on, up to the~root.
 We denote by $M$ the adjacency matrix of $G_t$ ordered by $\prec$.

 Let $M_T$ be the submatrix of $M$ obtained by deleting the $t$ rows and $t$ columns corresponding to $X$.
 Note that the grid number of $M$ is at most $\gn(M_T)+t$.
 We claim that there is no 3-grid minor in $M_T$.
 
 Indeed, in the order $\prec$, above the diagonal of $M_T$ there is no pair of 1-entries in strictly decreasing positions.
 Thus overall there is no triple of 1-entries in strictly decreasing positions.
 Thus no 3-grid minor is possible in $M_T$.
\end{proof}


\begin{thebibliography}{10}

\bibitem{Ahn22}
Jungho Ahn, Kevin Hendrey, Donggyu Kim, and Sang{-}il Oum.
\newblock Bounds for the twin-width of graphs.
\newblock {\em CoRR}, abs/2110.03957, 2021.
\newblock URL: \url{https://arxiv.org/abs/2110.03957}, \href
  {http://arxiv.org/abs/2110.03957} {\path{arXiv:2110.03957}}.

\bibitem{Balaban21}
Jakub Balab{\'{a}}n and Petr Hlinen{\'{y}}.
\newblock Twin-width is linear in the poset width.
\newblock In Petr~A. Golovach and Meirav Zehavi, editors, {\em 16th
  International Symposium on Parameterized and Exact Computation, {IPEC} 2021,
  September 8-10, 2021, Lisbon, Portugal}, volume 214 of {\em LIPIcs}, pages
  6:1--6:13. Schloss Dagstuhl - Leibniz-Zentrum f{\"{u}}r Informatik, 2021.
\newblock \href {https://doi.org/10.4230/LIPIcs.IPEC.2021.6}
  {\path{doi:10.4230/LIPIcs.IPEC.2021.6}}.

\bibitem{Balaban22}
Jakub Balab{\'{a}}n, Petr Hlinen{\'{y}}, and Jan Jedelsk{\'{y}}.
\newblock Twin-width and transductions of proper k-mixed-thin graphs.
\newblock {\em CoRR}, abs/2202.12536, 2022.
\newblock URL: \url{https://arxiv.org/abs/2202.12536}, \href
  {http://arxiv.org/abs/2202.12536} {\path{arXiv:2202.12536}}.

\bibitem{Berge21}
Pierre Berg{\'{e}}, {\'{E}}douard Bonnet, and Hugues D{\'{e}}pr{\'{e}}s.
\newblock Deciding twin-width at most 4 is {NP}-complete.
\newblock {\em CoRR}, abs/2112.08953, 2021.
\newblock URL: \url{https://arxiv.org/abs/2112.08953}, \href
  {http://arxiv.org/abs/2112.08953} {\path{arXiv:2112.08953}}.

\bibitem{Bilu06}
Yonatan Bilu and Nathan Linial.
\newblock Lifts, discrepancy and nearly optimal spectral gap*.
\newblock {\em Combinatorica}, 26(5):495--519, 2006.
\newblock \href {https://doi.org/10.1007/s00493-006-0029-7}
  {\path{doi:10.1007/s00493-006-0029-7}}.

\bibitem{twin-width8}
{\'{E}}douard Bonnet, Dibyayan Chakraborty, Eun~Jung Kim, Noleen K{\"{o}}hler,
  Raul Lopes, and St{\'{e}}phan Thomass{\'{e}}.
\newblock Twin-width {VIII:} delineation and win-wins.
\newblock {\em CoRR}, abs/2204.00722, 2022.
\newblock \href {http://arxiv.org/abs/2204.00722} {\path{arXiv:2204.00722}},
  \href {https://doi.org/10.48550/arXiv.2204.00722}
  {\path{doi:10.48550/arXiv.2204.00722}}.

\bibitem{twin-width2}
{\'{E}}douard Bonnet, Colin Geniet, Eun~Jung Kim, St{\'{e}}phan Thomass{\'{e}},
  and R{\'{e}}mi Watrigant.
\newblock Twin-width {II:} small classes.
\newblock In {\em Proceedings of the 2021 ACM-SIAM Symposium on Discrete
  Algorithms (SODA)}, pages 1977--1996, 2021.
\newblock \href {https://doi.org/10.1137/1.9781611976465.118}
  {\path{doi:10.1137/1.9781611976465.118}}.

\bibitem{twin-width3}
{\'{E}}douard Bonnet, Colin Geniet, Eun~Jung Kim, St{\'{e}}phan Thomass{\'{e}},
  and R{\'{e}}mi Watrigant.
\newblock Twin-width {III:} max independent set, min dominating set, and
  coloring.
\newblock In Nikhil Bansal, Emanuela Merelli, and James Worrell, editors, {\em
  48th International Colloquium on Automata, Languages, and Programming,
  {ICALP} 2021, July 12-16, 2021, Glasgow, Scotland (Virtual Conference)},
  volume 198 of {\em LIPIcs}, pages 35:1--35:20. Schloss Dagstuhl -
  Leibniz-Zentrum f{\"{u}}r Informatik, 2021.
\newblock \href {https://doi.org/10.4230/LIPIcs.ICALP.2021.35}
  {\path{doi:10.4230/LIPIcs.ICALP.2021.35}}.

\bibitem{twin-width4}
{\'{E}}douard Bonnet, Ugo Giocanti, Patrice~Ossona de~Mendez, Pierre Simon,
  St{\'{e}}phan Thomass{\'{e}}, and Szymon Toru\'{n}czyk.
\newblock Twin-width {IV:} ordered graphs and matrices.
\newblock {\em CoRR}, abs/2102.03117, 2021, accepted at STOC 2022.
\newblock URL: \url{https://arxiv.org/abs/2102.03117}, \href
  {http://arxiv.org/abs/2102.03117} {\path{arXiv:2102.03117}}.

\bibitem{twin-width6}
{\'E}douard Bonnet, Eun~Jung Kim, Amadeus Reinald, and St{\'e}phan
  Thomass{\'e}.
\newblock Twin-width {VI:} the lens of contraction sequences.
\newblock In {\em Proceedings of the 2022 Annual ACM-SIAM Symposium on Discrete
  Algorithms (SODA)}, pages 1036--1056. SIAM, 2022.

\bibitem{tww-polyker}
{\'{E}}douard Bonnet, Eun~Jung Kim, Amadeus Reinald, St{\'{e}}phan
  Thomass{\'{e}}, and R{\'{e}}mi Watrigant.
\newblock Twin-width and polynomial kernels.
\newblock In Petr~A. Golovach and Meirav Zehavi, editors, {\em 16th
  International Symposium on Parameterized and Exact Computation, {IPEC} 2021,
  September 8-10, 2021, Lisbon, Portugal}, volume 214 of {\em LIPIcs}, pages
  10:1--10:16. Schloss Dagstuhl - Leibniz-Zentrum f{\"{u}}r Informatik, 2021.
\newblock \href {https://doi.org/10.4230/LIPIcs.IPEC.2021.10}
  {\path{doi:10.4230/LIPIcs.IPEC.2021.10}}.

\bibitem{twin-width1}
{\'{E}}douard Bonnet, Eun~Jung Kim, St{\'{e}}phan Thomass{\'{e}}, and
  R{\'{e}}mi Watrigant.
\newblock Twin-width {I:} tractable {FO} model checking.
\newblock {\em J. {ACM}}, 69(1):3:1--3:46, 2022.
\newblock \href {https://doi.org/10.1145/3486655} {\path{doi:10.1145/3486655}}.

\bibitem{reduced-bdw}
{\'{E}}douard Bonnet, O{-}joung Kwon, and David~R. Wood.
\newblock Reduced bandwidth: a qualitative strengthening of twin-width in
  minor-closed classes (and beyond).
\newblock {\em CoRR}, abs/2202.11858, 2022.
\newblock URL: \url{https://arxiv.org/abs/2202.11858}, \href
  {http://arxiv.org/abs/2202.11858} {\path{arXiv:2202.11858}}.

\bibitem{Jacob22}
Hugo Jacob and Marcin Pilipczuk.
\newblock Bounding twin-width for bounded-treewidth graphs, planar graphs, and
  bipartite graphs.
\newblock {\em CoRR}, abs/2201.09749, 2022.
\newblock URL: \url{https://arxiv.org/abs/2201.09749}.

\bibitem{Kratsch22}
Stefan Kratsch, Florian Nelles, and Alexandre Simon.
\newblock On triangle counting parameterized by twin-width.
\newblock {\em CoRR}, abs/2202.06708, 2022.
\newblock URL: \url{https://arxiv.org/abs/2202.06708}, \href
  {http://arxiv.org/abs/2202.06708} {\path{arXiv:2202.06708}}.

\bibitem{Pettersson22}
William Pettersson and John Sylvester.
\newblock Bounds on the twin-width of product graphs.
\newblock {\em CoRR}, abs/2202.11556, 2022.
\newblock URL: \url{https://arxiv.org/abs/2202.11556}, \href
  {http://arxiv.org/abs/2202.11556} {\path{arXiv:2202.11556}}.

\bibitem{PilipczukS22}
Michal Pilipczuk and Marek Sokolowski.
\newblock Graphs of bounded twin-width are quasi-polynomially
  {\(\chi\)}-bounded.
\newblock {\em CoRR}, abs/2202.07608, 2022.
\newblock URL: \url{https://arxiv.org/abs/2202.07608}, \href
  {http://arxiv.org/abs/2202.07608} {\path{arXiv:2202.07608}}.

\bibitem{PilipczukSZ22}
Michal Pilipczuk, Marek Sokolowski, and Anna Zych{-}Pawlewicz.
\newblock Compact representation for matrices of bounded twin-width.
\newblock In Petra Berenbrink and Benjamin Monmege, editors, {\em 39th
  International Symposium on Theoretical Aspects of Computer Science, {STACS}
  2022, March 15-18, 2022, Marseille, France (Virtual Conference)}, volume 219
  of {\em LIPIcs}, pages 52:1--52:14. Schloss Dagstuhl - Leibniz-Zentrum
  f{\"{u}}r Informatik, 2022.
\newblock \href {https://doi.org/10.4230/LIPIcs.STACS.2022.52}
  {\path{doi:10.4230/LIPIcs.STACS.2022.52}}.

\bibitem{Przybyszewski22}
Wojciech Przybyszewski.
\newblock {VC}-density and abstract cell decomposition for edge relation in
  graphs of bounded twin-width, 2022.
\newblock URL: \url{https://arxiv.org/abs/2202.04006}, \href
  {https://doi.org/10.48550/ARXIV.2202.04006}
  {\path{doi:10.48550/ARXIV.2202.04006}}.

\bibitem{Schidler21}
Andr{\'{e}} Schidler and Stefan Szeider.
\newblock A {SAT} approach to twin-width.
\newblock {\em CoRR}, abs/2110.06146, 2021, accepted at ALENEX 2022.
\newblock URL: \url{https://arxiv.org/abs/2110.06146}, \href
  {http://arxiv.org/abs/2110.06146} {\path{arXiv:2110.06146}}.

\end{thebibliography}
\end{document}